\documentclass[a4paper,10pt,leqno,oneside]{amsart}
\usepackage{amsmath}
\usepackage{amsfonts}
\usepackage{amsthm}
\usepackage{bm}
\usepackage{mathtools}
\usepackage{amssymb}
\usepackage[top=20truemm, bottom=20truemm, left=15truemm, right=15truemm]{geometry}
\usepackage{enumitem}
\usepackage{color}
\makeatletter \@addtoreset{equation}{section} \makeatother
\newcommand{\eref}[1]{(\ref{#1})}
\newcommand{\tref}[1]{Theorem \ref{#1}}
\newcommand{\pref}[1]{Proposition \ref{#1}}

\newcommand{\sref}[1]{Section \ref{#1}}

\newcommand{\jump}[1]{[\![#1]\!]}

\theoremstyle{plain} \newtheorem{thm}{Theorem}[section] \newtheorem{lem}{Lemma}[section] \newtheorem{prop}{Proposition}[section] 
\theoremstyle{definition} \newtheorem{rem}{Remark}[section] \newtheorem{defi}{Definition}[section]
\title[Signorini-type and Tresca friction conditions]{Unique solvability of a crack problem with Signorini-type and Tresca friction conditions in a linearized elastodynamic body}
\author[Hiromichi Itou]{Hiromichi Itou}
\address{Department of Mathematics, Tokyo University of Science, 1-3 Kagurazaka, Shinjuku-ku, Tokyo 162-8601, Japan}
\email{h-itou@rs.tus.ac.jp}
\author[Takahito Kashiwabara]{Takahito Kashiwabara}
\address{Graduate School of Mathematical Sciences, The University of Tokyo, 3-8-1 Komaba, Meguro-ku, Tokyo 153-8914, Japan}
\email{tkashiwa@ms.u-tokyo.ac.jp}

\date{\today}
\makeatletter
\@namedef{subjclassname@2020}{\textup{2020} Mathematics Subject Classification}
\makeatother

\begin{document}
\begin{abstract}
	We consider dynamic motion of a linearized elastic body with a crack subject to a modified contact law, which we call \emph{the Signorini contact condition of dynamic type}, and to the Tresca friction condition.
	Whereas the modified contact law involves both displacement and velocity, it formally includes the usual non-penetration condition as a special case.
	We prove that there exists a unique strong solution to this model.
	It is remarkable that not only existence but also uniqueness is obtained and that no viscosity term that serves as a parabolic regularization is added in our model.
\end{abstract}
\maketitle

\section{Introduction}
Analysis of crack motion is one of the most important topics in fracture mechanics and it has also attracted much attention in material science or in seismology (e.g.\ \cite{Bro1999, Fre1990, UMB2014}).
However, at least from mathematical point of view, it is far from being understood because of highly nonlinear and singular behavior of cracks.
Even if we put aside problems regarding crack propagation, which are difficult even at the stage of modeling and will not be addressed in this paper, there still remain many mathematical difficulties as explained below.

In the static case, one of the basic models is linear elasticity with interfacial conditions representing the non-penetration contact law (also known as the Signorini condition) and the Coulomb friction law on the crack (see \cite{IKT2011, KhlKov2000}).
Its dynamical version, however, becomes much more difficult and no mathematical results on this problem seem to have been obtained.
For related problems, in which some conditions mentioned above are modified or simplified, there are several known studies.

First, for the wave equation with the Signorini condition, unique solvability is established for the halfspace in \cite{LebSch1984}.
Existence of a weak solution for general domains is proved by \cite{Kim1989}, but uniqueness remains open.
Generalization of these results to the linear elasticity equations are also unsolved.
If the Kelvin--Voigt viscoelastic model, in which a term serving as parabolic regularization is added to linear elasticity, is considered instead, then existence of a weak solution is obtained by e.g.\ \cite{EJK2005, Tani2020} and references therein.
If the contact law is furthermore modified in such a way that the Signorini condition is imposed on velocity rather than on displacement, then uniqueness of a weak solution is shown as well (see \cite[Section 4.4.2]{EJK2005}).

Second, dynamic friction problems also exhibit a difficulty.
In case of the Tresca friction law, where the threshold parameter of the tangential traction is a given function $g$, \cite{DuvLio1976} established unique solvability of the linear elasticity equations (without contact conditions) under the assumption that $g$ does not depend on the time variable. This result was extended to the time-dependent $g$ in our previous paper \cite{ItoKas2021}.
If the Coulomb friction law which is considered to be more realistic but is more complex is employed, \cite[Chapter 5]{EJK2005} proves existence of a solution to the Kelvin--Voigt viscoelastic model combined with the Signorini condition in velocity.
In the context of crack problems, \cite{CocSca2006, Tani2020} constructed a weak solution of the Kelvin--Voigt viscoelastic model with the Signorini condition in displacement and with the nonlocal (approximated) Coulomb friction law.

Namely, when the contact condition is imposed on displacement and is combined with some friction law, only existence of a solution is established in the presence of viscosity terms.
In view of such a situation, one would like to mathematically explore a dynamic elasticity model with contact and friction having the following properties:
\begin{enumerate}[label=(\roman{*})]
	\item classical linear elasticity is exploited without viscosity;
	\item not only existence but also uniqueness of a solution is ensured;
	\item contact law is formulated in terms of displacement, which is considered to be more realistic.
\end{enumerate}
In this paper, we propose to impose a contact condition to linear combination of normal displacement and normal velocity on the interface with some constant coefficient $\delta > 0$; see \eref{eq: modified Signorini}$_1$ below.
Since $\delta = 0$ and $\delta = \infty$ correspond to the contact conditions in displacement and in velocity, respectively, it can be regarded as an intermediate between them.
We call (2.2)$_1$ \emph{the Signorini contact condition of dynamic type} (hereinafter, referred to as \emph{SCD condition}).
With the SCD and Tresca friction conditions, we prove unique existence of a strong solution for the linear elastodyanmic equations, thus having properties (i) and (ii).
Moreover, property (iii) is also approached by our model because $\delta > 0$ can be fixed to an arbitrarily small value (however it is not possible to make exactly $\delta = 0$).

An expository interpretation of our result may be that making the Signorini contact condition in displacement ``dynamic a bit'' (recall that boundary conditions having quantities with time derivative are called dynamic) leads to some stabilization effect to the system.
We expect that this fact has some connection with Baumgarte-like stabilization techniques known in numerical simulations of non-smooth mechanics (see \cite{KikBro2017}), which is to be investigated in future.
The present result will also be of basic interest when we make an attempt to more involved crack problems, e.g., propagation and singular behavior of crack tips.

This paper is organized as follows.
In \sref{sec2}, we introduce notation and precise mathematical setting to be studied.
In \sref{sec3}, a variational inequality formulation as well as the definition of a strong solution is introduced, and we present the main theorem.
\sref{sec4} is devoted to its proof based on regularization of a variational inequality and Galerkin's method.
The strategy basically follows our previous study \cite{ItoKas2021}; nevertheless, the analysis, in particular a priori estimates and a uniqueness proof, becomes more intricate to deal with the contact condition.

\section{Preliminaries} \label{sec2}
\subsection{Notation}
Let $\Omega \subset \mathbb R^3$ be a bounded domain with a smooth boundary $\partial\Omega$ consisting of two parts $\Gamma_D \neq \emptyset$ and $\Gamma_N$ which are mutually disjoint.
Let $\Gamma$ be a two-dimensional closed smooth interface which separates $\Omega$ into two subdomains $\Omega_{\pm}$, that is,
\begin{equation*}
	\Omega = \Omega_+ \cup \Omega_-, \quad \Gamma = \overline\Omega_+ \cap \overline\Omega_-.
\end{equation*}
We assume that $\partial\Omega_{\pm}$ satisfy the Lipschitz condition and that $\partial\Omega_{\pm} \cap \Gamma_D \neq \emptyset$.
A crack is supposed to be represented by an open subset $\Gamma_c$ of $\Gamma$ such that $\overline\Gamma_c \subset \Gamma \setminus\partial\Gamma$ (namely, $\Gamma_c \Subset \Gamma$); we refer to $\Omega_c := \Omega \setminus \overline\Gamma_c$ as \emph{the domain with a crack}.
The unit normal vector associated to $\partial\Omega$ is denoted by $\bm n_{\partial\Omega}$, and the unit normal vector on $\Gamma$ pointing from $\Omega_-$ to $\Omega_+$ is denoted by $\bm n$.
The geometric situation explained so far is schematically summarized in Figure \ref{fig1}.

\begin{figure}[htbp]
	\centering
	\includegraphics[width=8cm]{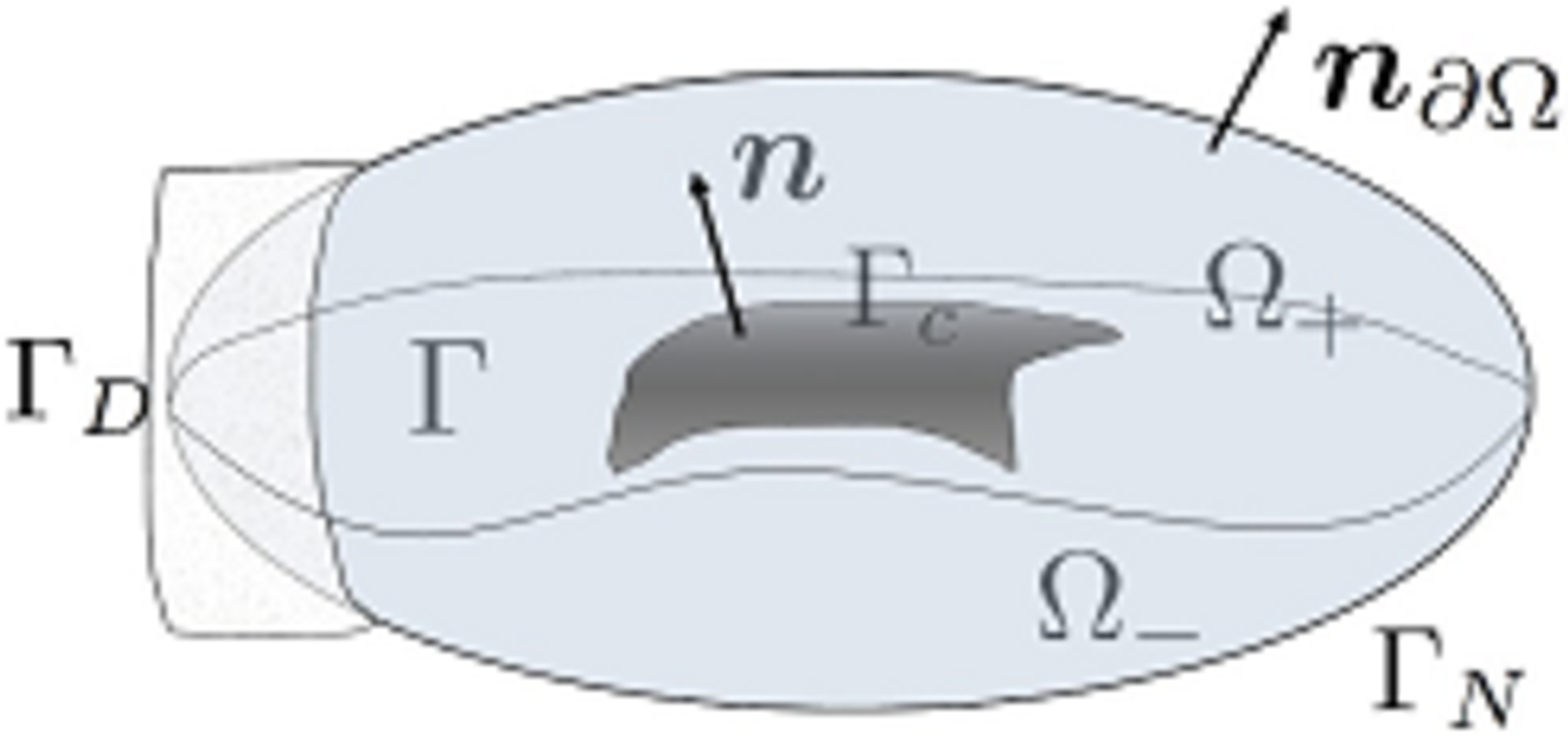}
	\caption{domain with a crack}
	\label{fig1}
\end{figure}

We mainly deal with functions defined in $\Omega_c$ in this paper.
For such a function $u$, we let $u^{\pm} := u|_{\Omega_{\pm}}$ be its restrictions to subdomains $\Omega_{\pm}$.
If $u^{\pm}$ are smooth enough, we define the jump discontinuity of $u$ across $\Gamma$ by
\begin{equation*}
	\jump{u} := u^{+}|_\Gamma - u^{-}|_\Gamma,
\end{equation*}
and that of $\nabla u$ by $\jump{\nabla u} := (\nabla u^{+})|_\Gamma - (\nabla u^{-})|_\Gamma$.

For function spaces, we employ the usual Lebesgue spaces $L^p(\Omega_c) \, (1\le p\le \infty)$ and the Sobolev space $H^1(\Omega_c)$, which have the characterization
\begin{align*}
	L^p(\Omega_c) &= L^p(\Omega_{+}) \times L^p(\Omega_{-}), \\
	H^1(\Omega_c) &= \{ (u^{+}, u^{-}) \in H^1(\Omega_{+}) \times H^1(\Omega_{-}) : \jump{u} = 0 \text{ on } \Gamma\setminus\Gamma_c \}.
\end{align*}
Accordingly, their norms are given by $\|u\|_{L^p(\Omega_c)} := (\|u^{+}\|_{L^p(\Omega_{+})}^p + \|u^{-}\|_{L^p(\Omega_{-})}^p)^{1/p}$ and $\|u\|_{H^1(\Omega_c)} := (\|u^{+}\|_{H^1(\Omega_{+})}^2 + \|u^{-}\|_{H^1(\Omega_{-})}^2)^{1/2}$.
Note in particular that if $u \in H^1(\Omega_c)$ then $\jump{u} \in H^{1/2}_{00}(\Gamma_c)$, which is the Lions--Magenes space (see \cite{LioMag1972}).

Functions and function spaces which are vector- or tensor-valued are written with bold fonts, e.g., $\bm u \in \bm H^1(\Omega_c) = H^1(\Omega_c)^3$, whereas fine fonts mean scalar quantities.
We denote the inner products of $L^2(\Omega_c)$ by $(\cdot, \cdot)$, and those of $L^2(\Gamma_N), L^2(\Gamma_c)$ by $(\cdot, \cdot)_{\Gamma_N}, (\cdot, \cdot)_{\Gamma_c}$ (the same notation will also be used for vectors and tensors).
We also exploit the notation of Bochner spaces $L^p(0, T; X)$ and $W^{k, p}(0, T; X)$ for a positive constant $T$ and a Banach space $X$, where $k>0$ is an integer and $1\le p\le \infty$.
Finally, the dual space of $X$ is denoted by $X^*$.

\subsection{Problem formulation}
We assume that $\Omega_c$ is regarded as a reference configuration (or non-deformed state) of an elastic body.
The deformation of the body may be described by a displacement field $\bm u : (0, T) \times \Omega_c \to \mathbb R^3$.
If the constitutive law of the material is based on isotropic linear elasticity, the stress tensor is given by
\begin{equation} \label{eq: constitutive law}
	\bm\sigma(\bm u) = (\lambda \operatorname{div} \bm u) \mathbb I + 2\mu \, \mathbb E(\bm u),
\end{equation}
where $\lambda, \mu$ are Lam\'e constants such that $\mu > 0$ and $3\lambda + 2\mu > 0$, $\mathbb I$ is the unit tensor, and $\mathbb E(\bm u) = (\nabla\bm u + (\nabla \bm u)^\top)/2$ means the linearized strain tensor.
The dynamic deformation of the body is governed by the hyperbolic system
\begin{equation*}
	\rho \bm u'' - \operatorname{div} \bm \sigma(\bm u) = \rho \bm f \quad\text{in}\quad (0, T) \times \Omega_c,
\end{equation*}
where $\rho$ is the density which is a positive constant, the prime stands for the time derivative (i.e., $\bm u'' = \partial_t^2 \bm u$), $\bm f$ is the external body force, and $T > 0$ stands for a fixed time length.
As for the boundary conditions, we consider 
\begin{align*}
	\bm u &= \bm 0 \quad\text{on}\quad (0, T) \times \Gamma_D, \\
	\bm\sigma(\bm u) \bm n_{\partial\Omega} &= \bm F \quad\text{on}\quad (0, T) \times \Gamma_N,
\end{align*}
where $\bm F$ is a prescribed traction on $\Gamma_N$.
At $t = 0$, the initial displacement and velocity fields are given as
\begin{equation*}
	\bm u(0) = \bm u_0, \quad \bm u'(0) = \dot{\bm u}_0 \quad\text{on}\quad \{0\}\times\Omega_c.
\end{equation*}

Before stating the interface conditions on the crack, we introduce the normal and tangential components of the displacement, velocity, and traction on $\Gamma$, restricted from $\Omega_{\pm}$, by
\begin{align*}
	u^{\pm}_n &= \bm u^{\pm} \cdot \bm n, \qquad \bm u^{\pm}_\tau = \bm u^{\pm} - u^{\pm}_n\bm n, \qquad
	u'^{\pm}_n = \bm u'^{\pm} \cdot \bm n, \qquad \bm u'^{\pm}_\tau = \bm u'^{\pm} - u'^{\pm}_n\bm n, \\
	\sigma^{\pm}_n &= \bm\sigma(\bm u^{\pm}) \bm n \cdot \bm n, \qquad \bm\sigma^{\pm}_\tau = \bm\sigma(\bm u^{\pm}) \bm n - \sigma_n^{\pm} \bm n,
\end{align*}
together with their jumps
\begin{align*}
	&\jump{u_n} = u^{+}_n - u^{-}_n, \quad \jump{\bm u_\tau} = \bm u^{+}_\tau - \bm u^{-}_\tau, \quad
	\jump{u'_n} = u'^{+}_n - u'^{-}_n, \quad \jump{\bm u'_\tau} = \bm u'^{+}_\tau - \bm u'^{-}_\tau, \\
	&\jump{\sigma_n} = \jump{\sigma_n(\bm u)} = \sigma^{+}_n - \sigma^{-}_n, \quad \jump{\bm\sigma_\tau} = \jump{\bm\sigma_\tau(\bm u)} = \bm\sigma^{+}_\tau - \bm\sigma^{-}_\tau.
\end{align*}
In this paper, we consider the Signorini contact condition of dynamic type (SCD condition) and Tresca friction condition on the crack $\Gamma_c$ as follows:
\begin{equation} \label{eq: modified Signorini}
\begin{aligned}
	\jump{\sigma_n} &= 0, \quad \sigma_n \le 0, \quad \jump{u_n + \delta u_n'} \ge 0, \quad \sigma_n \jump{u_n + \delta u_n'} = 0 && \text{on}\quad (0, T) \times \Gamma_c, \\
	\jump{\bm\sigma_\tau} &= \bm 0, \quad |\bm\sigma_\tau| \le g, \quad -\bm\sigma_\tau \cdot \jump{\bm u_\tau'} + g |\jump{\bm u_\tau'}| = 0 && \text{on}\quad (0, T) \times \Gamma_c,
\end{aligned}
\end{equation}
where $\delta \in (0, \infty]$ is a constant, $g = g(t, \bm x) \ge 0$ is a given function.

Several remarks are in order.
First, $\sigma_n := \sigma^{+}_n = \sigma^{-}_n$ and $\bm\sigma_\tau := \bm\sigma^{+}_{\tau} = \bm\sigma^{-}_{\tau}$ are well-defined as single-valued functions on $\Gamma_c$ because they have no jump by \eref{eq: modified Signorini}.
Second, if $\delta = 0$ in \eref{eq: modified Signorini}$_1$ then we formally recover the usual non-penetration condition introduced in \cite{KhlKov2000}.
On the other hand, if $\delta = \infty$ then we arrive at the contact condition in terms of velocity given by \cite{EJK2005}.
To see this we equivalently rewrite \eref{eq: modified Signorini}$_1$, with $\gamma := \delta^{-1}$, as
\begin{equation} \label{eq2: modified Signorini}
	\sigma_n \le 0, \quad \jump{\gamma u_n + u_n'} \ge 0, \quad \sigma_n \jump{\gamma u_n + u_n'} = 0,
\end{equation}
and set $\gamma = 0$.
For simplicity of presentation, we mainly deal with the SCD condition in the form \eref{eq2: modified Signorini} with $\gamma\in [0, \infty)$ rather than \eref{eq: modified Signorini}$_1$ in the subsequent analysis.
\begin{rem}
	The introduction of $\delta$ in \eref{eq: modified Signorini}$_1$ is mainly due to the mathematical reason as explained in Introduction.
	From a modeling viewpoint, it can be regarded as a first-order approximation to the usual non-penetration condition $\jump{u_n} \ge 0$.
	We see that the SCD condition allows for interpenetration of the crack, which is not physically feasible and may be a restriction in applications.
	However, it remains realistic for a short time interval in the case of no initial slip velocity on the crack (e.g., for the first---and usually strongest---wave of an earthquake as mentioned in \cite[Chapter 5]{EJK2005}).
\end{rem}

\section{Variational formulations} \label{sec3}
\subsection{Variational inequality}
As discussed in the previous section, the strong form of the initial boundary value problem considered in this paper is represented as follows:
\begin{equation} \label{eq: strong form}
\begin{aligned}
	\rho \bm u'' - \operatorname{div} \bm \sigma(\bm u) &= \rho \bm f && \text{in}\quad (0, T) \times \Omega_c, \\
	\bm u &= \bm 0 && \text{on}\quad (0, T) \times \Gamma_D, \\
	\bm\sigma(\bm u) \bm n_{\partial\Omega} &= \bm F && \text{on}\quad (0, T) \times \Gamma_N, \\
	\jump{\sigma_n} = 0, \quad \sigma_n \le 0, \quad \jump{\gamma u_n + u_n'} &\ge 0, \quad \sigma_n \jump{\gamma u_n + u_n'} = 0 && \text{on}\quad (0, T) \times \Gamma_c, \\
	\jump{\bm\sigma_\tau} = \bm0, \quad |\bm\sigma_\tau| &\le g, \quad \bm\sigma_\tau \cdot \jump{\bm u_\tau'} = g |\jump{\bm u_\tau'}| && \text{on}\quad (0, T) \times \Gamma_c, \\
	\bm u(0) = \bm u_0, \quad \bm u'(0) &= \dot{\bm u}_0 && \text{on}\quad \{0\} \times \Omega_c.
\end{aligned}
\end{equation}
Let us derive a weak formulation to this problem assuming that $\bm u$ is smooth enough in $[0, T] \times (\overline\Omega\setminus\Gamma_c)$.
To this end we introduce the following function spaces and convex cone:
\begin{equation*}
	\bm H := \bm L^2(\Omega_c), \quad
	\bm V := \{ \bm v \in \bm H^1(\Omega_c) \,:\, \bm v = \bm 0 \text{ on } \Gamma_D \}, \quad
	\bm K := \{ \bm v \in \bm V \,:\, \jump{v_n} \ge 0 \text{ a.e.\ on } \Gamma_c \}.
\end{equation*}
Multiplying \eref{eq: strong form}$_1$ by $\bm v - (\gamma\bm u + \bm u')$ with an arbitrary $\bm v \in \bm K$ and integrating over $\Omega_c$, we obtain
\begin{align*}
	&\rho \big( \bm u''(t), \bm v - (\gamma\bm u(t) + \bm u'(t)) \big)
	+ \big( \bm\sigma(\bm u(t)), \nabla(\bm v - (\gamma\bm u(t) + \bm u'(t))) \big) \\
	&\qquad + \big( \sigma_n(t), \jump{v_n - (\gamma u_n(t) + u'_n(t))} \big)_{\Gamma_c}
	+ \big( \bm\sigma_\tau(t), \jump{\bm v_\tau - (\gamma\bm u_\tau(t) + \bm u'_\tau(t))} \big)_{\Gamma_c} \\
	= \; &\rho \big( \bm f(t), \bm v - (\gamma\bm u(t) + \bm u'(t)) \big)
	+ \big( \bm F(t), \bm v - (\gamma\bm u(t) + \bm u'(t)) \big)_{\Gamma_N}
	\qquad \forall t \in (0, T),
\end{align*}
where we have used $\jump{\sigma_n} = 0, \jump{\bm\sigma_\tau} = \bm0$ on $\Gamma_c$ and the fact that the outer unit normal w.r.t.\ $\Omega_{\pm}$ on $\Gamma$ is $\mp\bm n$.
By \eref{eq: constitutive law} we see that
\begin{equation*}
	( \bm\sigma(\bm u), \nabla\bm v )
	= ( \bm\sigma(\bm u), \mathbb E(\bm v) )
	= \lambda (\operatorname{div} \bm u, \operatorname{div} \bm v) + 2\mu (\mathbb E(\bm u), \mathbb E(\bm v)) =: a(\bm u, \bm v) \quad \forall\bm v \in \bm V.
\end{equation*}
It follows from \eref{eq: strong form}$_4$ and \eref{eq: strong form}$_5$ that
\begin{align}
	\big( \sigma_n(\bm u(t)), \jump{v_n - (\gamma u_n(t) + u'_n(t))} \big)_{\Gamma_c} &\le 0, \label{eq: VI for sigma n} \\
	\big( \bm\sigma_\tau((\bm u(t))), \jump{\bm v_\tau - (\gamma\bm u_\tau(t) + \bm u'_\tau(t))} \big)_{\Gamma_c}
	&\le \big( g(t), |\jump{\bm v_\tau - \gamma\bm u_\tau(t)}| - |\jump{\bm u'_\tau(t)}| \big)_{\Gamma_c}. \label{eq: VI for sigma tau}
\end{align}
Consequently,
\begin{equation} \label{eq: VI}
\begin{aligned}
	&\rho \big( \bm u''(t), \bm v - (\gamma\bm u(t) + \bm u'(t)) \big) + a\big( \bm u(t), \bm v - (\gamma\bm u(t) + \bm u'(t)) \big)
		+ \big( g(t), |\jump{\bm v_\tau - \gamma\bm u_\tau(t)}| - |\jump{\bm u'_\tau(t)}| \big)_{\Gamma_c} \\
	\ge \;& \rho \big( \bm f(t), \bm v - (\gamma\bm u(t) + \bm u'(t)) \big)
		+ \big( \bm F(t), \bm v - \gamma\bm u(t) + \bm u'(t)) \big)_{\Gamma_N}
		\qquad \forall \bm v \in \bm V, \quad \text{a.e.\ } t \in (0, T).
\end{aligned}
\end{equation}

This is a variational inequality of hyperbolic type which is equivalent to the strong form \eref{eq: strong form} as seen below.
\begin{prop} \label{prop: equivalence}
	Let $\bm u : [0, T] \times (\overline\Omega\setminus\Gamma_c) \to \mathbb R^3$ be smooth enough.
	Then $\bm u$ solves \eref{eq: strong form} if and only if the following hold:
	
	(i) $\bm u(t) \in \bm V$ for all $t \in (0, T)$;
	
	(ii) $\bm u(0) = \bm u_0$ and $\bm u'(0) = \dot{\bm u}_0$;
	
	(iii) $\gamma \bm u(t) + \bm u'(t) \in \bm K$ for all $t \in (0, T)$;
	
	(iv) $\bm u$ satisfies the hyperbolic variational inequality \eref{eq: VI}.
\end{prop}
\begin{proof}
	The proof is essentially similar to \cite[pp.\ 125--126]{ItoKas2021}.
	It suffices to show the ``if'' part.
	Taking a test function $\bm v = \pm\bm w + \gamma\bm u(t) + \bm u'(t)$ with arbitrary $\bm w \in \bm V$ such that $\jump{\bm w} = 0$ on $\Gamma$, one can reduce \eref{eq: VI} to
	\begin{equation*}
		\rho ( \bm u''(t), \bm w ) + a( \bm u(t), \bm w ) = \rho ( \bm f(t), \bm w ) + ( \bm F(t), \bm w )_{\Gamma_N},
	\end{equation*}
	which implies \eref{eq: strong form}$_1$, \eref{eq: strong form}$_3$, $\jump{\sigma_n} = 0$ on $\Gamma_c$, and $\jump{\bm\sigma_\tau} = \bm0$ on $\Gamma_c$.
	Then \eref{eq: VI for sigma n} and \eref{eq: VI for sigma tau} follow from integration by parts (note that each of $\jump{v_n}$ and $\jump{\bm v_\tau}$ can be chosen to an arbitrary smooth function independently).
	
	First we focus on \eref{eq: VI for sigma n}.
	Setting $\jump{v_n}$ to $0$ and $2\jump{\gamma u_n(t) + u'_n(t)}$ gives $( \sigma_n(t), \jump{\gamma u_n(t) + u'_n(t)} )_{\Gamma_c} = 0$.
	Therefore, $(\sigma_n(t), \jump{v_n} )_{\Gamma_c} \le 0$ for arbitrary $\jump{v_n} \ge 0$, which implies $\sigma_n(t) \le 0$ on $\Gamma_c$.
	This combined with $\jump{\gamma u_n(t) + u'_n(t)} \ge 0$ on $\Gamma_c$ deduces the last equality of \eref{eq: strong form}$_4$.
	
	Next, in \eref{eq: VI for sigma tau}, setting $\jump{\bm v_\tau}$ to $\jump{\gamma\bm u_\tau(t)}$ and $\jump{\gamma\bm u_\tau(t) + 2\bm u'_\tau(t)}$ gives $\big( \bm\sigma_\tau(t), \jump{\bm u'_\tau(t)} \big)_{\Gamma_c} = \big( g(t), |\jump{\bm u'_\tau(t)}| \big)_{\Gamma_c}$.
	Therefore, $(\bm\sigma_\tau(t), \jump{\bm v_\tau})_{\Gamma_c} \le (g(t), |\jump{\bm v_\tau}|)_{\Gamma_c}$ for arbitrary $\jump{\bm v_\tau}$, which implies $|\bm\sigma_\tau(t)| \le g(t)$ on $\Gamma_c$.
	Then the last equality of \eref{eq: strong form}$_5$ also follows.
	This proves that $\bm u$ solves \eref{eq: strong form}.
\end{proof}

\subsection{Main result}
In view of \pref{prop: equivalence}, let us define a solution of \eref{eq: strong form} based on its variational form.

\begin{defi}
	Given $\bm f, \bm F, g, \bm u_0, \dot{\bm u}_0$, we say that $\bm u \in W^{2,\infty}(0, T; \bm H) \cap W^{1, \infty}(0, T; \bm V)$ is a strong solution of \eref{eq: strong form} if $\bm u$ satisfies conditions (i)--(iv) in \pref{prop: equivalence}.
\end{defi}

\begin{rem}
	For second-order hyperbolic problems, one usually considers a weak solution in $W^{1, \infty}(0, T; \bm L^2(\Omega_c)) \cap L^\infty(0, T; \bm H^1(\Omega_c))$.
	However, this class would not be appropriate for dynamic elasticity problems with friction where the trace of velocity explicitly appears on an interface.
	We also note that in the Kelvin--Voigt viscoelastic case, a natural class of a weak solution becomes $W^{1, \infty}(0, T; \bm L^2(\Omega_c)) \cap H^1(0, T; \bm H^1(\Omega_c))$, avoiding this issue.
\end{rem}

Now we are ready to state our main result in this paper.
\begin{thm} \label{thm: main}
	Let $\gamma \in [0, \infty)$, $\bm f \in H^1(0, T; \bm H), \bm F \in H^2(0, T; \bm L^2(\Gamma_N))$, and let $g \in H^2(0, T; L^2(\Gamma_c))$ be non-negative.
	We assume that $\bm u_0 \in \bm V$, $\dot{\bm u}_0 \in \bm V$ and that they satisfy the following compatibility conditions:
	\begin{itemize}
		\item $-\operatorname{div} \bm\sigma (\bm u_0) \in \bm H$;
		\item $\bm\sigma(\bm u_0) \bm n_{\partial\Omega} = \bm F(0)$ on $\Gamma_N$;
		\item $\sigma_n(\bm u_0^{+}) = \sigma_n(\bm u_0^{-}) = 0$ and $\jump{\gamma u_{0n} + \dot{u}_{0n}} = 0$ on $\Gamma_c$;
		\item $\bm\sigma_\tau(\bm u_0^{+}) = \bm\sigma_\tau(\bm u_0^{-}) = \bm 0$ and $\jump{\dot{\bm u}_{0\tau}} = \bm 0$ on $\Gamma_c$.
	\end{itemize}
	Then there exists a unique strong solution of \eref{eq: strong form}.
\end{thm}

\begin{rem}
	Since $\bm u_0 \in \bm V$ satisfies $-\operatorname{div} \bm\sigma (\bm u_0^{\pm}) \in \bm L^2(\Omega_{\pm})$, initial tractions $\bm\sigma(\bm u_0^{\pm}) \bm n_{\partial\Omega}$ and $\bm\sigma(\bm u_0^{\pm}) \bm n$ are well-defined in $(\bm H^{1/2}_{00}(\Gamma_N))^*$ and $(\bm H^{1/2}_{00}(\Gamma_c))^*$, respectively.
	The third and fourth conditions above are stronger than just requiring that $\bm u_0$ and $\dot{\bm u}_0$ satisfy \eref{eq: strong form}$_4$ and \eref{eq: strong form}$_5$ at $t = 0$;
	however, we are not aware whether they can be weakened.
\end{rem}

\subsection{Regularized problem}
It is not easy to directly construct a solution of the time-dependent variational inequality \eref{eq: VI} because it contains non-differentiable relations.
To see this, we introduce two convex functions
\begin{equation*}
	\psi(x) =
	\begin{cases}
		+\infty & (x < 0), \\
		0 & (x \ge 0),
	\end{cases}
	\qquad
	\varphi(\bm x) = |\bm x| \quad (\bm x \in \mathbb R^3),
\end{equation*}
whose subdifferentials $\beta := \partial\psi$ and $\bm\alpha := \partial\varphi$ are maximal monotone graphs given by
\begin{equation*}
	\beta(x) =
	\begin{cases}
		\emptyset & x < 0, \\
		(-\infty, 0] & x = 0, \\
		0 & x > 0,
	\end{cases}
	\qquad
	\bm\alpha(\bm x) =
	\begin{cases}
		\bm x/|\bm x| & (\bm x \neq \bm0), \\
		\{ \bm y \in \mathbb R^3 \,:\, |\bm y| \le 1 \} & (\bm x = \bm0).
	\end{cases}
\end{equation*}
We then observe that the SCD and Tresca conditions in \eref{eq: strong form} are concisely expressed as
\begin{equation} \label{eq: BC with subdifferential}
	\sigma_n \in \beta(\jump{\gamma u_n + u'_n}), \qquad \bm\sigma_\tau \in g \bm\alpha(\jump{\bm u_\tau'}).
\end{equation}

To address the difficulty that $\beta$ and $\bm\alpha$ are multi-valued functions and non-differentiable, we approximate $\psi$ and $\varphi$ by the following functions which are convex and $W^{3, \infty} \cap C^2$:
\begin{equation*}
	\psi_\epsilon(x) = \frac{1}{3\epsilon} [x]_{-}^3,
	\qquad
	\varphi_\epsilon(\bm x) = \sqrt{|\bm x|^2 + \epsilon^2},
\end{equation*}
where $\epsilon > 0$ is a constant and $[x]_{-} := \max\{-x, 0\}$ for $x \in \mathbb R$.
Their derivatives $\beta_\epsilon := \frac{d\psi_\epsilon}{dx}$ and $\alpha_\epsilon := \nabla \varphi_\epsilon$ are given by
\begin{equation*}
	\beta_\epsilon(x) = -\frac1\epsilon [x]_{-}^2, \qquad
	\bm\alpha_\epsilon(\bm x) = \frac{\bm x}{\sqrt{|\bm x|^2 + \epsilon^2}},
\end{equation*}
which are monotone and $W^{2, \infty} \cap C^1$.

With this preparation we consider the following regularized problem denoted by \textbf{(VI)$_\epsilon$}: find $\bm u_\epsilon(t) \in \bm V$ such that $\bm u_\epsilon(0) = \bm u_0, \bm u_\epsilon'(0) = \dot{\bm u}_0$ and
\begin{equation} \label{eq: VI epsilon}
	\begin{aligned}
		&\rho (\bm u_\epsilon''(t), \bm v - (\gamma \bm u_\epsilon(t) + \bm u_\epsilon'(t))) + a(\bm u_\epsilon(t), \bm v - (\gamma \bm u_\epsilon(t) + \bm u_\epsilon'(t))) \\
		&\qquad + \big( 1, \psi_\epsilon( \jump{v_n} ) - \psi_\epsilon( \jump{\gamma u_{\epsilon n}(t) + u_{\epsilon n}'(t)} ) \big)_{\Gamma_c} + \big( g(t), \varphi_\epsilon( \jump{\bm v_\tau - \gamma \bm u_{\epsilon\tau}(t)} ) - \varphi_\epsilon( \jump{\bm u_{\epsilon\tau}'(t)} ) \big)_{\Gamma_c} \\
		\ge \;& \rho(\bm f(t), \bm v - (\gamma \bm u_\epsilon(t) + \bm u_\epsilon'(t))) + (\bm F(t), \bm v - (\gamma \bm u_\epsilon(t) + \bm u_\epsilon'(t)))_{\Gamma_N} \qquad \forall \bm v \in \bm V, \quad \text{a.e.\ } t \in (0, T).
	\end{aligned}
\end{equation}

In the proposition below we find that \textbf{(VI)$_\epsilon$} is equivalent to the following variational equality problem denoted by \textbf{(VE)$_\epsilon$}: find $\bm u_\epsilon(t) \in \bm V$ such that $\bm u_\epsilon(0) = \bm u_0, \bm u_\epsilon'(0) = \dot{\bm u}_0$ and
\begin{equation} \label{eq: VE epsilon}
	\begin{aligned}
		&\rho (\bm u_\epsilon''(t), \bm v) + a(\bm u_\epsilon(t), \bm v) + \big( \beta_\epsilon( \jump{\gamma u_{\epsilon n}(t) + u_{\epsilon n}'(t)}), \jump{v_n} \big)_{\Gamma_c} + \big( g(t) \bm\alpha_\epsilon( \jump{\bm u_{\epsilon\tau}'(t)}), \jump{\bm v_\tau} \big)_{\Gamma_c} \\
		= \; & \rho(\bm f(t), \bm v) + (\bm F(t), \bm v)_{\Gamma_N} \qquad \forall \bm v \in \bm V, \quad \text{a.e.\ } t \in (0, T).
	\end{aligned}
\end{equation}

\begin{prop} \label{prop: VI and VE are equivalent}
	Let $\bm u_\epsilon \in W^{2,\infty}(0, T; \bm H) \cap W^{1, \infty}(0, T; \bm V)$.
	It solves \textbf{(VI)$_\epsilon$} if and only if it solves \textbf{(VE)$_\epsilon$}.
\end{prop}
\begin{proof}
	Although the proof is standard, we present it for completeness.
	Let $\bm u_\epsilon$ be a solution of \textbf{(VI)$_\epsilon$}.
	Taking $\bm v = \pm h \bm w + \gamma \bm u_\epsilon(t) + \bm u_\epsilon'(t))$ with arbitrary $h > 0$ and $\bm w \in \bm V$, dividing by $h$, and letting $h \to 0$, we deduce \textbf{(VE)$_\epsilon$} from the relations
	\begin{align*}
		\lim_{h \to 0} \frac{ \psi_\epsilon(\jump{h w_n + \gamma u_{\epsilon n}(t) + u_{\epsilon n}'(t)}) - \psi_\epsilon(\jump{\gamma u_{\epsilon n}(t) + u_{\epsilon n}'(t)}) }{h} &= \beta_\epsilon(\jump{\gamma u_{\epsilon n}(t) + u_{\epsilon n}'(t)}) \, \jump{w_n}, \\
		\lim_{h \to 0} \frac{ \varphi_\epsilon(\jump{h \bm w_\tau + \bm u_{\epsilon\tau}'(t)}) - \varphi_\epsilon(\jump{\bm u_{\epsilon\tau}'(t)}) }{h} &= \bm\alpha_\epsilon(\jump{\bm u_{\epsilon\tau}'(t)}) \cdot \jump{\bm w_\tau}.
	\end{align*}
	
	Conversely, let $\bm u_\epsilon$ be a solution of \textbf{(VE)$_\epsilon$}.
	Notice that, since $\psi_\epsilon$ and $\varphi_\epsilon$ are convex,
	\begin{align*}
		\psi_\epsilon(\jump{w_n + \gamma u_{\epsilon n}(t) + u_{\epsilon n}'(t)}) - \psi_\epsilon(\jump{\gamma u_{\epsilon n}(t) + u_{\epsilon n}'(t)}) &\ge \beta_\epsilon(\jump{\gamma u_{\epsilon n}(t) + u_{\epsilon n}'(t)}) \, \jump{w_n}, \\
		\varphi_\epsilon(\jump{\bm w_\tau + \bm u_{\epsilon\tau}'(t)}) - \varphi_\epsilon(\jump{\bm u_{\epsilon\tau}'(t)}) &\ge \bm\alpha_\epsilon(\jump{\bm u_{\epsilon\tau}'(t)}) \cdot \jump{\bm w_\tau},
	\end{align*}
	for all $\bm w \in \bm V$.
	Setting this $\bm w$ in such a way that $\bm w + \gamma \bm u_{\epsilon}(t) + \bm u_{\epsilon}'(t) = \bm v$ and using \eref{eq: VE epsilon}, we arrive at \eref{eq: VI epsilon}.
\end{proof}

As a result of \pref{prop: VI and VE are equivalent}, it suffices to solve an equation problem for obtaining $\bm u_\epsilon$.
Furthermore, since it follows from \eref{eq: VE epsilon} that
\begin{equation*}
	\sigma_n(\bm u_\epsilon) = \beta_\epsilon(\jump{\gamma u_{\epsilon n} + u'_{\epsilon n}}), \quad \bm\sigma_\tau(\bm u_\epsilon) = \bm\alpha_\epsilon(\jump{\bm u_{\epsilon\tau}'}) \quad\text{on}\quad (0, T) \times \Gamma_c,
\end{equation*}
we expect that $\bm u_\epsilon$ should converge to a solution of the original problem \eref{eq: strong form} as $\epsilon \to 0$.
Justification of this fact, which is actually the idea to prove \tref{thm: main}, is the task of the next section.

\section{Proof of main result} \label{sec4}
We establish existence in Sections \ref{sec4.1}--\ref{sec4.4} and uniqueness in \sref{sec4.5}.
Coercivity of $a(\cdot, \cdot)$ in $\bm V$, that is, 
\begin{equation*}
	a(\bm v, \bm v) \ge C \|\bm v\|_{H^1(\Omega_c)}^2 \qquad \forall \bm v \in \bm V,
\end{equation*}
which is justified by Korn's inequality (see e.g.\ \cite{DuvLio1976}), will be frequently used in the proof.
Here and in what follows, $C$ represents a generic constant depending only on the domain $\Omega_c$, Lam\'e constants $\lambda, \mu$, and density $\rho$.
We will also write $C(\bm f, g)$ etc.\ in order to indicate dependency on other quantities.

The inequality above allows us to define the norm of $\bm V$ as $\|\bm v\|_{\bm V} := \bm a(\bm v, \bm v)^{1/2}$, whereas we use $\|\bm v\|_{\bm H} := \|\bm v\|_{\bm L^2(\Omega_c)}$.

\subsection{Galerkin approximation} \label{sec4.1}
We apply Galerkin's method to solve \eref{eq: VE epsilon}.
Since $\bm V \subset \bm H^1(\Omega_c)$ is separable, there exist countable members $\bm w_1, \bm w_2, \cdots \in \bm V$, which are linearly independent, such that $\overline{ \bigcup_{m=1}^\infty \bm V_m } = \bm V$ where $\bm V_m := \operatorname{span}\{\bm w_k\}_{k=1}^m$.
We may assume that $\bm u_0, \dot{\bm u}_0 \in \bm V_m$ for $m \ge 2$ (otherwise one can add $\bm u_0$ and $\dot{\bm u}_0$ to the members $\{\bm w_k\}_{k=1}^m$). 

For $m = 2, 3, \dots$, the Galerkin approximation problem consists in determining $c_k(t) \, (k=1, \dots, m)$ such that $\bm u_m = \sum_{k=1}^m c_k(t) \bm w_k(\bm x) \in \bm V_m$ satisfies
\begin{equation} \label{eq: Galerkin 1}
	\begin{aligned}
		&\rho (\bm u_m''(t), \bm v) + a(\bm u_m(t), \bm v) + \big( \beta_\epsilon( \jump{\gamma u_{m n}(t) + u_{m n}'(t)} ), \jump{v_n} \big)_{\Gamma_c} + \big( g(t) \bm\alpha_\epsilon( \jump{\bm u_{m\tau}'(t)} ), \jump{\bm v_\tau} \big)_{\Gamma_c} \\
		= \; & \rho(\bm f(t), \bm v) + (\bm F(t), \bm v)_{\Gamma_N} \qquad \forall \bm v \in \bm V_m, \quad \forall t \in (0, T),
	\end{aligned}
\end{equation}
together with the initial conditions $\bm u_m(0) = \bm u_0, \bm u_m'(0) = \dot{\bm u}_0$.

This is a finite-dimensional system of ODEs that admits a local-in-time unique solution $\{c_k \in W^{3, \infty}(0, \tilde T) \cap C^2([0, \tilde T])\}_{k=1}^m$ for certain $0 < \tilde T \le T$ (recall that $\beta_\epsilon, \alpha_\epsilon$ are $W^{2, \infty} \cap C^1$).
Because the \textit{a priori} estimates below ensure that $\tilde T$ can be extended to $T$, we use $T$ instead of $\tilde T$ from the beginning.

Differentiating \eref{eq: Galerkin 1} in $t$ we obtain
\begin{equation} \label{eq: Galerkin 2}
	\begin{aligned}
		&\rho (\bm u_m'''(t), \bm v) + a(\bm u_m'(t), \bm v) + \big( \beta_\epsilon'( \jump{\gamma u_{m n}(t) + u_{m n}'(t)} ) \jump{\gamma u_{m n}'(t) + u_{m n}''(t)}, \jump{v_n} \big)_{\Gamma_c} \\
		&\qquad + \big( g'(t) \bm\alpha_\epsilon( \jump{\bm u_{m\tau}'(t)} ), \jump{\bm v_\tau} \big)_{\Gamma_c} + \big( g(t) \nabla \bm\alpha_\epsilon( \jump{\bm u_{m\tau}'(t)}) \jump{\bm u_{m\tau}''(t)}, \jump{\bm v_\tau} \big)_{\Gamma_c} \\
		= \; & \rho(\bm f'(t), \bm v) + (\bm F'(t), \bm v)_{\Gamma_N} \qquad \forall \bm v \in \bm V_m, \quad \forall t \in (0, T).
	\end{aligned}
\end{equation}

\subsection{First \emph{a priori} estimate}
Let us establish an estimate for $\bm u_m \in W^{1, \infty}(0, T; \bm H) \cap L^\infty(0, T; \bm V)$.
For arbitrary $t \in (0, T)$ take $\bm v = \gamma \bm u_m + \bm u_m' \in \bm V_m$ in \eref{eq: Galerkin 1} to obtain
\begin{align*}
	&\frac12 \frac{d}{dt} \big( \rho \|\bm u_m'(t)\|_{\bm H}^2 + \|\bm u_m(t)\|_{\bm V}^2 \big) + \gamma \|\bm u_m(t)\|_{\bm V}^2 + \frac{1}{\epsilon} \big\| \jump{\gamma u_{m n}(t) + u_{m n}'(t)}_{-} \big\|_{L^3(\Gamma_c)}^3 + \rho \gamma (\bm u_m''(t), \bm u_m(t)) \\
	\le \;& \rho(\bm f(t), \gamma \bm u_m(t) + \bm u_m'(t)) + (\bm F(t), \gamma \bm u_m(t) + \bm u_m'(t))_{\Gamma_N} - \gamma \big( g(t) \bm\alpha_\epsilon( \jump{\bm u_{m\tau}'(t)} ), \jump{\bm u_{m\tau}(t)} \big)_{\Gamma_c},
\end{align*}
where $\jump{v}_{-}$ means $[\jump{v}]_{-}$ and we have used $\beta_\epsilon(x)x = \frac1\epsilon [x]_{-}^3, \bm\alpha_\epsilon(\bm x) \cdot \bm x \ge 0$.
Applying H\"older's and Young's inequalities to terms involving $\gamma$ on the right-hand side yields
\begin{align*}
	&\frac12 \frac{d}{dt} \big( \rho \|\bm u_m'(t)\|_{\bm H}^2 + \|\bm u_m(t)\|_{\bm V}^2 \big) + \frac\gamma2 \|\bm u_m(t)\|_{\bm V}^2 + \frac{1}{\epsilon} \big\| \jump{\gamma u_{m n}(t) + u_{m n}'(t)}_{-} \big\|_{L^3(\Gamma_c)}^3 + \rho \gamma (\bm u_m''(t), \bm u_m(t)) \\
	\le \;& C\gamma( \|\bm f(t)\|_{\bm H}^2 + \|\bm F(t)\|_{\bm L^2(\Gamma_N)}^2 + \|g(t)\|_{L^2(\Gamma_c)}^2) + \rho(\bm f(t), \bm u_m'(t)) + (\bm F(t), \bm u_m'(t))_{\Gamma_N},
\end{align*}
where we have used $|\bm\alpha_\epsilon(\cdot)| \le 1$ and the trace inequality $\|\jump{\bm v}\|_{\bm L^2(\Gamma_c)} \le C\|\bm v\|_{\bm V}$.
Integration of the both sides with respect to $t$ gives
\begin{align*}
	&\frac12 \big( \rho \|\bm u_m'(t)\|_{\bm H}^2 + \|\bm u_m(t)\|_{\bm V}^2 \big) + \frac\gamma2 \int_0^t \|\bm u_m(s)\|_{\bm V}^2 \, ds + \frac{1}{\epsilon} \int_0^t \big\| \jump{\gamma u_{m n}(s) + u_{m n}'(s)}_{-} \big\|_{L^3(\Gamma_c)}^3 \, ds \\
		&\qquad + \rho \gamma [(\bm u_m'(s), \bm u_m(s))]_{0}^{t} - \rho \gamma \int_0^t \|\bm u_m'(s)\|_{\bm H}^2 \, ds \\
	\le \;& \frac12 \big( \rho \|\dot{\bm u}_0\|_{\bm H}^2 + \|\bm u_0)\|_{\bm V}^2 \big)
		+ C\gamma( \|\bm f\|_{L^2(0, T; \bm H)}^2 + \|\bm F\|_{L^2(0, T; \bm L^2(\Gamma_N))}^2 + \|g\|_{L^2(0, T; L^2(\Gamma_c))}^2) \\
		&\qquad + \frac\rho2 \|\bm f\|_{L^2(0, T; \bm H)}^2 + \frac\rho2 \int_0^t \|\bm u_m'(s)\|_{\bm H}^2 \, ds
		+ [(\bm F(s), \bm u_m(s))_{\Gamma_N}]_{0}^{t} - \int_0^t (\bm F'(s), \bm u_m(s))_{\Gamma_N} \, ds.
\end{align*}
In particular,
\begin{align*}
	&\rho \|\bm u_m'(t)\|_{\bm H}^2 + \frac12 \|\bm u_m(t)\|_{\bm V}^2 + \rho \gamma \frac{d}{dt} \|\bm u_m(t)\|_{\bm H}^2
	+ \frac{2}{\epsilon} \int_0^t \big\| \jump{\gamma u_{m n}(s) + u_{m n}'(s)}_{-} \big\|_{L^3(\Gamma_c)}^3 \, ds \\
	\le \;& C (\gamma + 1) ( \|\bm f\|_{L^2(0, T; \bm H)}^2 + \|\bm F\|_{H^1(0, T; \bm L^2(\Gamma_N))}^2 + \|g\|_{L^2(0, T; L^2(\Gamma_c))}^2+ \|\bm u_0\|_{\bm V}^2 + \|\dot{\bm u}_0\|_{\bm H}^2) \\
	&\qquad + C(\gamma + 1) \int_0^t \Big( \rho \|\bm u_m'(s)\|_{\bm H}^2 + \frac12 \|\bm u_m(s)\|_{\bm V}^2 \Big) \, ds,
\end{align*}
where $(\bm F(t), \bm u_m(t))_{\Gamma_N}$ has been bounded by $C\|\bm F\|_{H^1(0, T; \bm L^2(\Gamma_N))}^2 + \frac14 \|\bm u_m(t)\|_{\bm V}^2$.
Setting $A(t) := \rho \|\bm u_m'(t)\|_{\bm H}^2 + \frac12 \|\bm u_m(t)\|_{\bm V}^2$ and neglecting the last term on the left-hand side (this is just for simplicity of presentation; if we keep this term, we obtain \eref{eq2: a priori estimate 1} below), we rephrase this estimate as
\begin{equation} \label{eq1: a priori estimate 1}
	A(t) + \rho \gamma \frac{d}{dt} \|\bm u_m(t)\|_{\bm H}^2 \le C_1(\bm f, \bm F, g, \bm u_0, \dot{\bm u}_0)(\gamma + 1) + C(\gamma + 1) \int_0^t A(s) \, ds \qquad \forall t \in (0, T). 
\end{equation}

If $\gamma = 0$, we find from Gronwall's inequality that
\begin{equation*}
	A(t) \le C_1(\bm f, \bm F, g, \bm u_0, \dot{\bm u}_0) e^{Ct}.
\end{equation*}
Otherwise we further integrate \eref{eq1: a priori estimate 1} with respect to $t$, with $B_1(t) := \int_0^t A(s) \, ds$, to get
\begin{equation*}
	 B_1(t) + \rho \gamma \|\bm u_m(t)\|_{\bm H}^2 \le C_2(\bm f, \bm F, g, \bm u_0, \dot{\bm u}_0, T)(\gamma + 1) + C(\gamma + 1 ) \int_0^t B_1(s) \, ds,
\end{equation*}
so that, by Gronwall's inequality,
\begin{equation*}
	B_1(t) + \rho \gamma \|\bm u_m(t)\|_{\bm H}^2 \le C_2(\bm f, \bm F, g, \bm u_0, \dot{\bm u}_0, T)(\gamma + 1) e^{C(\gamma + 1) t}.
\end{equation*}
Since $\rho \gamma \frac{d}{dt} \|\bm u_m(t)\|_{\bm H}^2 = 2\rho \gamma (\bm u_m'(t), \bm u_m(t))$, we find from \eref{eq1: a priori estimate 1} that
\begin{equation*}
	A(t) \le C_1(\bm f, \bm F, g, \bm u_0, \dot{\bm u}_0)(\gamma + 1) + C(\gamma + 1) B_1(t) + \frac{\rho}{2} \|\bm u_m'(t)\|_{\bm H}^2 + 2 \rho \gamma^2 \|\bm u_m(t)\|_{\bm H}^2,
\end{equation*}
which concludes
\begin{equation} \label{eq: conclusion of a priori estimate 1}
	\frac12 \big( \rho\|\bm u_m'(t)\|_{\bm H}^2 + \|\bm u_m(t)\|_{\bm V}^2 \big) \le C_3(\bm f, \bm F, g, \bm u_0, \dot{\bm u}_0, T) (\gamma + 1)^2 e^{C(\gamma + 1) t}.
\end{equation}

\begin{rem}
	As we already noticed before \eref{eq1: a priori estimate 1}, it also holds that, for all $t \in [0, T]$,
	\begin{equation} \label{eq2: a priori estimate 1}
		\frac{2}{\epsilon} \int_0^t \big\| \jump{\gamma u_{m n}(s) + u_{m n}'(s)}_{-} \big\|_{L^3(\Gamma_c)}^3 \, ds \le C_3(\bm f, \bm F, g, \bm u_0, \dot{\bm u}_0, T) (\gamma + 1)^2 e^{C(\gamma + 1) t}.
	\end{equation}
\end{rem}

\subsection{Second \emph{a priori} estimate}
Next let us establish an estimate for $\bm u_m' \in W^{1, \infty}(0, T; \bm H) \cap L^\infty(0, T; \bm V)$.
For arbitrary $t \in (0, T)$ we take $\bm v = \gamma \bm u_{m}' + \bm u_m'' \in \bm V_m$ in \eref{eq: Galerkin 2} to obtain
\begin{align*}
	&\frac12 \frac{d}{dt} \big( \rho \|\bm u_m''(t)\|_{\bm H}^2 + \|\bm u_m'(t)\|_{\bm V}^2 \big) + \gamma \|\bm u_m'(t)\|_{\bm V}^2 + \gamma \rho (\bm u_m'''(t), \bm u_m'(t)) \\
	\le \;& \gamma \rho (\bm f'(t), \bm u_m'(t)) + \gamma (\bm F'(t), \bm u_m'(t))_{\Gamma_N}
		+ \rho (\bm f'(t), \bm u_m''(t)) + (\bm F'(t), \bm u_m''(t))_{\Gamma_N} \\
		&\quad - \gamma \big( g(t) \nabla \bm\alpha_\epsilon( \jump{\bm u_{m\tau}'(t)}) \jump{\bm u_{m\tau}''(t)}, \jump{\bm u_{m\tau}'(t)} \big)_{\Gamma_c}
		- \gamma \big( g'(t) \bm\alpha_\epsilon( \jump{\bm u_{m\tau}'(t)} ), \jump{\bm u_{m\tau}'} \big)_{\Gamma_c}
		- \big( g'(t) \bm\alpha_\epsilon( \jump{\bm u_{m\tau}'(t)} ), \jump{\bm u_{m\tau}''} \big)_{\Gamma_c},
\end{align*}
where we have used the fact that $\beta_\epsilon'$ and $\nabla\bm\alpha_\epsilon$ are non-negative.

Applying H\"older's and Young's inequalities to the first three and the sixth terms on the right-hand side, together with $|\bm\alpha_\epsilon(\cdot)| \le 1$ and the trace inequality $\|\jump{\bm v}\|_{\bm L^2(\Gamma_c)} \le C\|\bm v\|_{\bm V}$, we have
\begin{align*}
	&\frac12 \frac{d}{dt} \big( \rho \|\bm u_m''(t)\|_{\bm H}^2 + \|\bm u_m'(t)\|_{\bm V}^2 \big) + \frac\gamma2 \|\bm u_m'(t)\|_{\bm V}^2 + \rho \gamma (\bm u_m'''(t), \bm u_m'(t)) \\
	\le \;& C(\gamma + 1) (\|\bm f'(t)\|_{\bm H}^2 + \|\bm F'(t)\|_{L^2(\Gamma_N)}^2 + \|g'(t)\|_{L^2(\Gamma_c)}^2) + \rho\|\bm u_m''(t)\|_{\bm H}^2 \\
	&\qquad + (\bm F'(t), \bm u_m''(t))_{\Gamma_N} - \gamma \Big( g(t) \frac{d}{dt} \bm\alpha_\epsilon( \jump{\bm u_{m\tau}'(t)}), \jump{\bm u_{m\tau}'(t)} \Big)_{\Gamma_c} - \Big( g'(t), \frac{d}{dt} \varphi_\epsilon( \jump{\bm u_{m\tau}'(t)} ) \Big)_{\Gamma_c}.
\end{align*}
Integration of the both sides with respect to $t$ yields
\begin{align*}
	&\frac12 ( \rho \|\bm u_m''(t)\|_{\bm H}^2 + \|\bm u_m'(t)\|_{\bm V}^2 ) + \frac\gamma2 \int_0^t \|\bm u_m'(s)\|_{\bm V}^2 \, ds
		+ \rho \gamma \big[ (\bm u_m''(s), \bm u_m'(s)) \big]_0^t - \rho \gamma \int_0^t \|\bm u_m''(s)\|_{\bm H}^2 \, ds \\
	\le \;& \frac12 (\rho \|\bm u_m''(0)\|_{\bm H}^2 + \|\dot{\bm u}_0\|_{\bm V}^2)
		+ C(\gamma + 1) (\|\bm f'\|_{L^2(0, T; \bm H)}^2 + \|\bm F'\|_{L^2(0, T; \bm L^2(\Gamma_N))}^2 + \|g'\|_{L^2(0, T; L^2(\Gamma_c))}^2)
		+ \int_0^t \rho\|\bm u_m''(s)\|_{\bm H}^2 \, ds \\
	&\quad + \big[ (\bm F'(s), \bm u_m'(s))_{\Gamma_N} \big]_0^t - \int_0^t (\bm F''(s), \bm u_m'(s))_{\Gamma_N} \, ds \\
	&\quad - \gamma \Big[ \big( g(s) \bm\alpha_\epsilon( \jump{\bm u_{m\tau}'(s)}), \jump{\bm u_{m\tau}'(s)} \big)_{\Gamma_c} \Big]_0^t
		+ \gamma \int_0^t \big( g'(s) \bm\alpha_\epsilon( \jump{\bm u_{m\tau}'(s)}), \jump{\bm u_{m\tau}'(s)} \big)_{\Gamma_c} \, ds \\
	& \hspace{6.3cm} + \gamma \int_0^t \big( g(s), \underbrace{ \bm\alpha_\epsilon( \jump{\bm u_{m\tau}'(s)}) \jump{\bm u_{m\tau}''(s)} }_{ = \frac{d}{ds} \varphi_\epsilon(\jump{\bm u_{m\tau}'(s)}) } \big)_{\Gamma_c} \, ds \\[-2mm]
	&\quad - \big[ \big( g'(s), \varphi_\epsilon( \jump{\bm u_{m\tau}'(s)} \big)_{\Gamma_c} \big]_0^t + \int_0^t \big( g''(s), \varphi_\epsilon( \jump{\bm u_{m\tau}'(s)} \big)_{\Gamma_c} \, ds,
\end{align*}
where the eighth term on the right-hand side equals
\begin{align*}
	\gamma \big[ \big( g(s), \varphi_\epsilon(\jump{\bm u_{m\tau}'(s)}) \big)_{\Gamma_c} \big]_0^t - \gamma \int_0^t \big( g'(s), \varphi_\epsilon(\jump{\bm u_{m\tau}'(s)}) \big)_{\Gamma_c} \, ds.
\end{align*}
H\"older's and Young's inequalities, combined with the relations $H^1(0, T; L^2(\Gamma_c)) \hookrightarrow C([0, T]; L^2(\Gamma_c))$, $\|\jump{\bm v}\|_{\bm L^2(\Gamma_c)} \le C\|\bm v\|_{\bm V}$ and with $|\bm\alpha_\epsilon(\cdot)| \le 1$, $\varphi_{\epsilon}(\cdot) = \sqrt{|\cdot|^2 + \epsilon^2}$, lead to
\begin{equation} \label{eq1: a priori estimate 2}
\begin{aligned}
	&\rho \|\bm u_m''(t)\|_{\bm H}^2 + \frac12 \|\bm u_m'(t)\|_{\bm V}^2 + \gamma \int_0^t \|\bm u_m'(s)\|_{\bm V}^2 \, ds
		+ \rho \gamma \frac{d}{dt} \|\bm u_m'(t)\|_{\bm H}^2 \\
	\le \;& C(\gamma + 1) \big( \|\bm u_m''(0)\|_{\bm H}^2 + \|\bm f\|_{H^1(0, T; \bm H)}^2 + \|\bm F\|_{H^2(0, T; \bm L^2(\Gamma_N))}^2 + \|g\|_{H^2(0, T; L^2(\Gamma_c))}^2 + \|\dot{\bm u}_0\|_{\bm V}^2 + \epsilon^2 \big) \\
	& + C(\gamma + 1) \int_0^t (\rho\|\bm u_m''(s)\|_{\bm H}^2 + \frac12 \|\bm u_m'(s)\|_{\bm V}^2) \, ds + C\gamma^2 \|g\|_{H^1(0, T; L^2(\Gamma_c))}^2 \qquad \forall t \in (0, T),
\end{aligned}
\end{equation}
where the last contribution owes to $\gamma (g(t) \bm\alpha_\epsilon( \jump{\bm u_{m\tau}'(t)}), \jump{\bm u_{m\tau}'(t)})_{\Gamma_c}$ and $\gamma ( g(t), \varphi_\epsilon(\jump{\bm u_{m\tau}'(t)}) )_{\Gamma_c}$.

It remains to estimate $\|\bm u''(0)\|_{\bm H}$.
For this purpose we make $t = 0$ and take $\bm v = \bm u''(0) \in \bm V_m$ in \eref{eq: Galerkin 1} to see
\begin{align*}
	&\rho \|\bm u_m''(0)\|_{\bm H}^2 + a(\bm u_0, \bm u_m''(0)) + \big( \beta_\epsilon(\jump{ \gamma u_{0n} + \dot{u}_{0n} }), \jump{u_{mn}''(0)} \big)_{\Gamma_c}
		+ \big( g(0) \bm\alpha_\epsilon( \jump{\dot{\bm u}_{0\tau}}), \jump{\bm u_{m\tau}''(0)} \big)_{\Gamma_c} \\
	= \;& (\rho \bm f(0), \bm u_m''(0)) + (\bm F(0), \bm u_m''(0))_{\Gamma_N}.
\end{align*}
Noting that
\begin{align*}
	a(\bm u_0, \bm u_m''(0)) = (-\operatorname{div} \bm\sigma(\bm u_0), \bm u_m''(0)) + \big( \bm\sigma(\bm u_0)\bm n, \bm u_m''(0) \big)_{\Gamma_N}
		- (\sigma_n(\bm u_0), \jump{u_{mn}''(0)})_{\Gamma_c} - (\bm\sigma_\tau(\bm u_0), \jump{\bm u_{m\tau}''(0)})_{\Gamma_c}
\end{align*}
and using the compatibility conditions, we deduce
\begin{equation*}
	\rho \|\bm u_m''(0)\|_{\bm H}^2 = \big( \operatorname{div} \bm\sigma(\bm u_0) + \rho \bm f(0), \bm u_m''(0) \big),
\end{equation*}
which implies $\|\bm u_m''(0)\|_{\bm H} \le C(\|\operatorname{div} \bm\sigma(\bm u_0)\|_{\bm H} + \|\bm f(0)\|_{\bm H})$.

Substituting this into \eref{eq1: a priori estimate 2}, we proceed as in the previous subsection assuming $\epsilon \le 1$.
If $\gamma = 0$, Gronwall's inequality gives us
\begin{equation*}
	\rho \|\bm u_m''(t)\|_{\bm H}^2 + \frac12 \|\bm u_m'(t)\|_{\bm V}^2 \le C_4(\bm f, \bm F, g, \bm u_0, \dot{\bm u}_0) e^{Ct}.
\end{equation*}
If $\gamma > 0$, we further integrate \eref{eq1: a priori estimate 2} to have
\begin{equation*}
	B_2(t) + \rho \gamma \|\bm u_m'(t)\|_{\bm H}^2 \le C_5(\bm f, \bm F, g, \bm u_0, \dot{\bm u}_0, T) (\gamma+1)^2 + C(\gamma + 1)\int_0^t B_2(s) \, ds,
\end{equation*}
where $B_2(t) := \int_0^t (\rho \|\bm u_m''(t)\|_{\bm H}^2 + \frac12 \|\bm u_m'(t)\|_{\bm V}^2) \, ds$.
Applying Gronwall's inequality above and substituting the resulting estimate into \eref{eq1: a priori estimate 2}, in which $2\rho \gamma |(\bm u_m''(t), \bm u_m'(t))|$ is bounded by $\frac\rho2 \|\bm u_m''(t)\|_{\bm H}^2 + 2\rho \gamma^2 \|\bm u_m'(t)\|_{\bm H}^2$, we conclude
\begin{equation} \label{eq: conclusion of a priori estimate 2}
	\rho \|\bm u_m''(t)\|_{\bm H}^2 + \|\bm u_m'(t)\|_{\bm V}^2 \le C_6(\bm f, \bm F, g, \bm u_0, \dot{\bm u}_0, T) (\gamma + 1)^3 e^{C(\gamma + 1)t}.
\end{equation}

\subsection{Passage to limit} \label{sec4.4}
The argument of the passage to the limits $m \to \infty$ and $\epsilon \to 0$ is basically similar to \cite[Section 3.7]{ItoKas2021}, the essential difference lying in the verification of the constraint $\gamma\bm u(t) + \bm u'(t) \in \bm K$.
However, for the sake of completeness we present the whole proof.

First let us consider the limit $m \to \infty$ for fixed $\epsilon \in (0, 1]$.
As a consequence of the \emph{a priori} estimates \eref{eq: conclusion of a priori estimate 1} and \eref{eq: conclusion of a priori estimate 2}, there exist a subsequence of $\{\bm u_m\}$, denoted by the same symbol, and some $\bm u_\epsilon \in W^{2,\infty}(0, T; \bm H) \cap W^{1, \infty}(0, T; \bm V)$ such that
\begin{align*}
	\bm u_{m} \rightharpoonup \bm u_\epsilon &\quad\text{weakly-$*$ in}\quad L^\infty(0, T; \bm V), \\
	\bm u_{m}' \rightharpoonup \bm u_\epsilon' &\quad\text{weakly-$*$ in}\quad L^\infty(0, T; \bm V), \\
	\bm u_{m}'' \rightharpoonup \bm u_\epsilon'' &\quad\text{weakly-$*$ in}\quad L^\infty(0, T; \bm H),
\end{align*}
as $m \to \infty$.
Here, we notice the compact embedding $W^{1, \infty}(0, T; L^2(\Omega_{\pm})) \cap L^\infty(0, T; H^1(\Omega_{\pm})) \hookrightarrow C([0, T]; L^2(\Omega_{\pm}))$ (see \cite{Simon1986})
and the compactness of the trace operator $H^1(\Omega_{\pm}) \to L^3(\Gamma_c)$ (see e.g.\ \cite{Nec12}).
It then follows that
\begin{align}
	\bm u_m \to \bm u_\epsilon \quad\text{and}\quad \bm u_m' \to \bm u_\epsilon' &\quad\text{strongly in}\quad C([0, T]; \bm H), \notag \\
	\jump{\bm u_m} \to \jump{\bm u_\epsilon} \quad\text{and}\quad \jump{\bm u_m'} \to \jump{\bm u_\epsilon'} &\quad\text{strongly in}\quad C([0, T]; \bm L^3(\Gamma_c)), \label{eq1: passge}
\end{align}
as $m \to \infty$.
In particular, the initial conditions $\bm u_\epsilon(0) = \bm u_0$ and $\bm u_\epsilon'(0) = \dot{\bm u}_0$ hold.
By choosing a further subsequence, we may also assume that
\begin{equation*}
	\jump{\bm u_m} \to \jump{\bm u_\epsilon} \quad\text{and}\quad \jump{\bm u_m'} \to \jump{\bm u_\epsilon'} \quad\text{a.e.\ in}\quad (0, T) \times \Gamma_c.
\end{equation*}

For arbitrary $\eta \in C_0^\infty(0, T)$ and $\bm v \in \bm V_m \, (m = 2, 3, \dots)$, we find from \eref{eq: Galerkin 1} that
\begin{align*}
	&\int_0^T \eta(t) \Big(
		\rho (\bm u_m''(t), \bm v) + a(\bm u_m(t), \bm v)
		+ \big( \beta_\epsilon( \jump{\gamma u_{m n}(t) + u_{m n}'(t)} ), \jump{v_n} \big)_{\Gamma_c}
		+ \big( g(t) \bm\alpha_\epsilon( \jump{\bm u_{m\tau}'(t)} ), \jump{\bm v_\tau} \big)_{\Gamma_c} \\
		&\hspace{3cm} - \rho(\bm f(t), \bm v) - (\bm F(t), \bm v)_{\Gamma_N}
	\Big) \, dt = 0.
\end{align*}
Letting $m \to \infty$, using \eref{eq1: passge}, and applying the dominated convergence theorem, we have
\begin{align*}
	&\int_0^T \eta(t) \Big(
		\rho (\bm u_\epsilon''(t), \bm v) + a(\bm u_\epsilon(t), \bm v)
		+ \big( \beta_\epsilon( \jump{\gamma u_{\epsilon n}(t) + u_{\epsilon n}'(t)} ), \jump{v_n} \big)_{\Gamma_c}
		+ \big( g(t) \bm\alpha_\epsilon( \jump{\bm u_{\epsilon\tau}'(t)} ), \jump{\bm v_\tau} \big)_{\Gamma_c} \\
		&\hspace{3cm} - \rho(\bm f(t), \bm v) - (\bm F(t), \bm v)_{\Gamma_N}
	\Big) \, dt = 0.
\end{align*}
Since $\overline{ \bigcup_{m=1}^\infty \bm V_m } = \bm V$ and $\eta$ is arbitrary, we conclude \eref{eq: VE epsilon}, that is, $\bm u_\epsilon$ is a solution of \textbf{(VE)$_\epsilon$} and also of \textbf{(VI)$_\epsilon$} by virtue of \pref{prop: VI and VE are equivalent}.
Moreover, by making $m \to \infty$ in \eref{eq: conclusion of a priori estimate 1}, \eref{eq2: a priori estimate 1}, and \eref{eq: conclusion of a priori estimate 2}, we also obtain
\begin{equation} \label{eq2: passage}
	\|\bm u_\epsilon\|_{W^{2, \infty}(0, T; \bm H)}^2 + \|\bm u_\epsilon\|_{W^{1, \infty}(0, T; \bm V)}^2
		+ \frac{1}{\epsilon} \int_0^T \big\| \jump{\gamma u_{\epsilon n}(s) + u_{\epsilon n}'(s)}_{-} \big\|_{L^3(\Gamma_c)}^3 \, ds 
	\le C(\bm f, \bm F, g, \bm u_0, \dot{\bm u}_0, T, \gamma).
\end{equation}

Next we consider the limit $\epsilon \to 0$.
By \eref{eq2: passage}, there exist a subsequence of $\{\bm u_\epsilon\}$, denoted by the same symbol, and some $\bm u \in W^{2,\infty}(0, T; \bm H) \cap W^{1, \infty}(0, T; \bm V)$ such that
\begin{align*}
	\bm u_\epsilon \rightharpoonup \bm u &\quad\text{weakly-$*$ in}\quad L^\infty(0, T; \bm V), \\
	\bm u_\epsilon' \rightharpoonup \bm u' &\quad\text{weakly-$*$ in}\quad L^\infty(0, T; \bm V), \\
	\bm u_\epsilon'' \rightharpoonup \bm u'' &\quad\text{weakly-$*$ in}\quad L^\infty(0, T; \bm H),
\end{align*}
as $\epsilon \to \infty$.
We observe from the third term on the left-hand side of \eref{eq2: passage} that
\begin{equation*}
	\int_0^T \big\| \jump{\gamma u_{n}(t) + u_{n}'(t)}_{-} \big\|_{L^3(\Gamma_c)}^3 \, dt 
		= \lim_{\epsilon\to0} \int_0^T \big\| \jump{\gamma u_{\epsilon n}(t) + u_{\epsilon n}'(t)}_{-} \big\|_{L^3(\Gamma_c)}^3 \, dt = 0,
\end{equation*}
which verifies $\jump{\gamma u_{n}(t) + u_{n}'(t)} \ge 0$ a.e.\ on $(0, T) \times \Gamma_c$, that is, $\gamma\bm u(t) + \bm u'(t) \in \bm K$ for $t \in (0, T)$.

For arbitrary $\tilde{\bm v} \in L^2(0, T; \bm K)$ we find from \eref{eq: VI epsilon} that
\begin{align*}
	&\int_0^T \Big( \rho (\bm u_\epsilon'', \tilde{\bm v} - (\gamma \bm u_\epsilon + \bm u_\epsilon')) + a(\bm u_\epsilon, \tilde{\bm v} - (\gamma \bm u_\epsilon + \bm u_\epsilon'))
	+ \big( g, \varphi_\epsilon( \jump{\tilde{\bm v}_\tau - \gamma \bm u_{\epsilon\tau}} ) - \varphi_\epsilon( \jump{\bm u_{\epsilon\tau}'} ) \big)_{\Gamma_c} \\
	&\qquad -\rho(\bm f, \tilde{\bm v} - (\gamma \bm u_\epsilon + \bm u_\epsilon')) - (\bm F, \tilde{\bm v} - (\gamma \bm u_\epsilon + \bm u_\epsilon'))_{\Gamma_N} \Big) \, dt \ge 0,
\end{align*}
because $\psi_\epsilon( \jump{\tilde v_n(t)} ) = 0$ and $\psi_\epsilon( \jump{\gamma u_{\epsilon n}(t) + u_{\epsilon n}'(t)} ) \ge 0$.
Consequently, 
\begin{equation} \label{eq3: passage}
\begin{aligned}
	&\int_0^T \Big( \rho (\bm u_\epsilon'', \tilde{\bm v} - \gamma \bm u_\epsilon) + a(\bm u_\epsilon, \tilde{\bm v})
	+ \big( g, \varphi_\epsilon( \jump{\tilde{\bm v}_\tau - \gamma \bm u_{\epsilon\tau}} ) - \varphi_\epsilon( \jump{\bm u_{\epsilon\tau}'} ) \big)_{\Gamma_c} \\
	&\hspace{5cm} -\rho(\bm f, \tilde{\bm v} - (\gamma \bm u_\epsilon + \bm u_\epsilon')) - (\bm F, \tilde{\bm v} - (\gamma \bm u_\epsilon + \bm u_\epsilon'))_{\Gamma_N} \Big) \, dt \\
	\ge \; &\int_0^T \big( \rho(\bm u_\epsilon'', \bm u_\epsilon') + a(\bm u_\epsilon, \bm u_\epsilon') + \gamma a(\bm u_\epsilon, \bm u_\epsilon) \big) \, dt \\
	= \; & \frac12 (\rho \|\bm u_\epsilon'(T)\|_{\bm H}^2 + \|\bm u_\epsilon(T)\|_{\bm V}^2)
		- \frac12 (\rho \|\dot{\bm u}_0\|_{\bm H}^2 + \|\bm u_0\|_{\bm V}^2)
		+ \gamma \|\bm u_\epsilon\|_{L^2(0, T; \bm V)}^2.
\end{aligned}
\end{equation}
Here, observe that $\|\bm u'(T)\|_{\bm H}^2 = \lim_{\epsilon\to0} \|\bm u_\epsilon'(T)\|_{\bm H}^2$ (recall the compact embedding $W^{1, \infty}(0, T; L^2(\Omega_{\pm})) \cap L^\infty(0, T; H^1(\Omega_{\pm})) \hookrightarrow C([0, T]; L^2(\Omega_{\pm}))$),
that $\varphi_\epsilon( \jump{\bm u_{\epsilon\tau}'} ) \to |\jump{\bm u_{\tau}'}|$ in $C([0, T]; L^2(\Gamma_c))$ as $\epsilon \to 0$, 
and that
\begin{equation*}
	\|\bm u(T)\|_{\bm V}^2 \le \liminf_{\epsilon\to0} \|\bm u_\epsilon(T)\|_{\bm V}^2, \qquad
	\|\bm u\|_{L^2(0, T; \bm V)}^2 \le \liminf_{\epsilon\to0} \|\bm u_\epsilon\|_{L^2(0, T; \bm V)}^2.
\end{equation*}
The former inequality above results from the following weak convergence:
\begin{equation*}
	a(\bm u(T) - \bm u_\epsilon(T), \bm w) = \int_0^T \big( a(\bm u'(t) - \bm u_\epsilon'(t), \eta(t)\bm w) + a(\bm u(t) - \bm u_\epsilon(t), \eta'(t)\bm w) \big) \, dt \to 0 \quad \forall \bm w \in \bm V, \; \epsilon \to 0,
\end{equation*}
where $\eta \in C^\infty([0, \infty])$ is chosen so that $\eta(0) = 0$ and $\eta(T) = 1$.
Therefore, making $\epsilon \to 0$ in \eref{eq3: passage} deduces
\begin{align*}
	&\int_0^T \Big( \rho (\bm u'', \tilde{\bm v} - \gamma \bm u) + a(\bm u, \tilde{\bm v})
		+ \big( g, |\jump{\tilde{\bm v}_\tau - \gamma \bm u_{\tau}}| ) - |\jump{\bm u_{\tau}'}| \big)_{\Gamma_c}
		-\rho(\bm f, \tilde{\bm v} - (\gamma \bm u + \bm u')) - (\bm F, \tilde{\bm v} - (\gamma \bm u + \bm u'))_{\Gamma_N} \Big) \, dt \\
	\ge \; &\int_0^T \big( \rho(\bm u'', \bm u') + a(\bm u, \bm u') + \gamma a(\bm u, \bm u) \big) \, dt,
\end{align*}
namely,
\begin{align*}
	&\int_0^T \Big( \rho (\bm u'', \tilde{\bm v} - (\gamma \bm u + \bm u')) + a(\bm u, \tilde{\bm v} - (\gamma \bm u + \bm u'))
	+ \big( g, |\jump{\tilde{\bm v}_\tau - \gamma \bm u_{\tau}}| ) - |\jump{\bm u_{\tau}'}| \big)_{\Gamma_c} \\
	&\qquad -\rho(\bm f, \tilde{\bm v} - (\gamma \bm u + \bm u')) - (\bm F, \tilde{\bm v} - (\gamma \bm u + \bm u'))_{\Gamma_N} \Big) \, dt \ge 0.
\end{align*}
This implies the pointwise (in time) variational inequality \eref{eq: VI} by a technique based on the Lebesgue differentiation theorem (see \cite[pp.\ 57--58]{DuvLio1976}).
Thus the existence part of \tref{thm: main} has been established.

\subsection{Uniqueness} \label{sec4.5}
Before proceeding to the proof of the uniqueness part of \tref{thm: main}, we present some preparatory results.
\begin{lem}
	There exists a vector function $\bm N \in \bm H^1(\Omega)$ such that its trace satisfies
	\begin{equation*}
		\bm N = \bm n \;\text{ on }\; \Gamma_c, \quad \bm N = \bm 0 \;\text{ on }\; \partial\Omega.
	\end{equation*}
\end{lem}
\begin{proof}
	Let $\tilde\Gamma_c$ be a neighborhood of $\overline\Gamma_c$ such that $\Gamma_c \Subset \tilde\Gamma_c \Subset \Gamma$.
	Then there exists $\tilde{\bm n} \in \bm H^{1/2}_{00}(\Gamma)$ such that $\tilde{\bm n} = \bm n$ on $\Gamma_c$ and $\tilde{\bm n} = \bm 0$ on $\Gamma \setminus \tilde\Gamma_c$.
	Then one can find some $\bm N_{\pm} \in \bm H^1(\Omega_{\pm})$ whose trace to $\partial\Omega_{\pm}$ equals the zero extension of $\tilde{\bm n}$ to $\partial\Omega_{\pm}$.
	If we define $\bm N = \bm N^{+}$ in $\Omega_{+}$ and $\bm N = \bm N^{-}$ in $\Omega_{-}$, this is a desired function.
\end{proof}
Using this lemma we introduce, for $\bm v \in \bm V$,
\begin{equation*}
	\bar{\bm v} := \bm v - (\bm v\cdot\bm N)\bm N.
\end{equation*}
Note that $\|\bar{\bm v}\|_{\bm H} \le C \|\bm v\|_{\bm H}$, $\|\bar{\bm v}\|_{\bm V} \le C\|\bm v\|_{\bm V}$, and that $\jump{\bar{v}_n} = 0$, $\jump{\bar{\bm v}_\tau} = \jump{\bm v_\tau}$, $\jump{((\bm v \cdot \bm N) \bm N)_\tau} = \bm 0$ on $\Gamma_c$.

For any solution $\bm u$ of \eref{eq: VI}, we see that $\sigma_n(\bm u) \in L^2(0, T; H^{1/2}_{00}(\Gamma_c)^*)$ and $\bm\sigma_\tau(\bm u) \in L^2(0, T; \bm H^{1/2}_{00}(\Gamma_c)^*)$ are characterized by
\begin{align*}
	\left< \sigma_n(\bm u(t)), \jump{v_n} \right>_{\Gamma_c} &= -\rho (\bm u''(t), \bm v) - a(\bm u(t), \bm v) + \rho(\bm f(t), \bm v) + (\bm F(t), \bm v)_{\Gamma_N} \qquad \forall \bm v \in \bm V, \; \jump{\bm v_\tau} = \bm 0 \text{ on } \Gamma_c, \\
	\left< \bm\sigma_\tau(\bm u(t)), \jump{\bm v_\tau} \right>_{\Gamma_c} &= -\rho (\bm u''(t), \bm v) - a(\bm u(t), \bm v) + \rho(\bm f(t), \bm v) + (\bm F(t), \bm v)_{\Gamma_N} \qquad \forall \bm v \in \bm V, \; \jump{v_n} = 0 \text{ on } \Gamma_c,
\end{align*}
respectively.
The next lemma is essentially a consequence of the monotonicity of $\beta$ and $\bm\alpha$ appearing in \eref{eq: BC with subdifferential}.
\begin{lem} \label{lem: monotonicity}
	If $\bm u_1, \bm u_2$ are two solutions of \eref{eq: VI}, then for a.e.\ $t \in (0, T)$
	\begin{align*}
		\left< \sigma_n(\bm u_1(t)) - \sigma_n(\bm u_2(t)), \jump{\gamma u_{1n}(t) + u_{1n}'(t)} - \jump{\gamma u_{2n}(t) + u_{2n}'(t)} \right>_{\Gamma_c} &\ge 0, \\
		\left< \bm\sigma_\tau(\bm u_1(t)) - \bm\sigma_\tau(\bm u_2(t)), \jump{\bm u_{1\tau}'(t)} - \jump{\bm u_{2\tau}'(t)} \right>_{\Gamma_c} &\ge 0.
	\end{align*}
\end{lem}
\begin{proof}
	Arguing in the same way as in \pref{prop: equivalence}, we get
	\begin{equation*}
		\left< \sigma_n(\bm u_i), \jump{v_n} \right>_{\Gamma_c} \le 0 \quad \forall \bm v \in \bm K, \quad\text{and}\quad
		\left< \sigma_n(\bm u_i), \jump{\gamma u_{in} + u_{in}'} \right>_{\Gamma_c} = 0,
	\end{equation*}
	for $i = 1, 2$.
	The first desired inequality follows from these and $\gamma\bm u_i + \bm u_i' \in \bm K$.
	
	Again by the same way as in \pref{prop: equivalence}, we have
	\begin{equation*}
		\left< \bm\sigma(\bm u_i), \jump{\bm v_\tau} \right>_{\Gamma_c} \le (g(t), |\jump{\bm v_\tau}|)_{\Gamma_c} \quad \forall \bm v \in \bm V, \quad\text{and}\quad
		\left< \bm\sigma_\tau(\bm u_i), \jump{\bm u_{i\tau}'} \right>_{\Gamma_c} = (g(t), |\jump{\bm u_{i\tau}'}|)_{\Gamma_c},
	\end{equation*}
	for $i = 1, 2$, which lead to the second desired inequality.
\end{proof}

Now we prove the uniqueness of a solution of \eref{eq: VI}.
Let $\bm u_1, \bm u_2$ be two solutions of \eref{eq: VI} and set $\bm w := \bm u_1 - \bm u_2$.
Then it follows that
\begin{equation*}
	\rho (\bm w''(t), \bm v) + a(\bm w(t), \bm v) + \left< \sigma_n(\bm w(t)), \jump{v_n} \right>_{\Gamma_c} + \left< \sigma_\tau(\bm w(t)), \jump{\bm v_\tau} \right>_{\Gamma_c} = 0 \quad \forall \bm v \in \bm V, \text{ a.e.\ } t \in (0, T).
\end{equation*}
Taking $\bm v = \gamma \bm w(t) + \bm w'(t)$ and using Lemma \ref{lem: monotonicity}, we deduce that
\begin{align*}
	\frac12 \frac{d}{dt} (\rho \|\bm w'(t)\|_{\bm H}^2 + \|\bm w(t)\|_{\bm V}^2) + \rho \gamma (\bm w''(t), \bm w(t)) + \gamma \|\bm w(t)\|_{\bm V}^2
		&\le -\gamma \left< \sigma_\tau(\bm w(t)), \jump{\bm w_\tau(t)} \right>_{\Gamma_c} \\
		&= -\gamma \left< \sigma_\tau(\bm w(t)), \jump{\bar{\bm w}_\tau(t)} \right>_{\Gamma_c} \\
		&= \rho \gamma (\bm w''(t), \bar{\bm w}(t)) + \gamma a(\bm w(t), \bar{\bm w}(t)),
\end{align*}
which, combined with $(\bm w''(t), \bm w(t) - \bar{\bm w}(t)) = (\bm w''(t) \cdot \bm N, \bm w(t) \cdot \bm N)$, gives
\begin{equation*}
	\frac12 \frac{d}{dt} (\rho \|\bm w'(t)\|_{\bm H}^2 + \|\bm w(t)\|_{\bm V}^2) + \rho \gamma (\bm w''(t) \cdot \bm N, \bm w(t) \cdot \bm N) \le \gamma a(\bm w(t), \bar{\bm w}(t)) \le C \gamma \|\bm w(t)\|_{\bm V}^2.
\end{equation*}
Integrate this with respect to $t$ to obtain (note that $\bm w(0) = \bm w'(0) = \bm 0$)
\begin{align*}
	\frac12 (\rho \|\bm w'(t)\|_{\bm H}^2 + \|\bm w(t)\|_{\bm V}^2) + \frac{\rho \gamma}{2} \frac{d}{dt} \|\bm w(t) \cdot \bm N\|_{\bm H}^2 &\le \gamma \int_0^t (\rho \|\bm w'(s) \cdot \bm N\|_{\bm H}^2 + C \|\bm w(s)\|_{\bm V}^2) \, ds \\
		&\le C \gamma \int_0^t (\rho \|\bm w'(s)\|_{\bm H}^2 + \|\bm w(s)\|_{\bm V}^2) \, ds.
\end{align*}
Setting $D(t) := \int_0^t (\rho \|\bm w'(s)\|_{\bm H}^2 + \|\bm w(s)\|_{\bm V}^2) \, ds$, we find from further integration of this estimate that
\begin{equation*}
	D(t) + \rho \gamma \|\bm w(t) \cdot \bm N\|_{\bm H}^2 \le C \gamma \int_0^t D(s) \, ds.
\end{equation*}
By Gronwall's inequality, $D(t) \equiv 0$ and hence $\bm w(t) \equiv \bm 0$, which shows the uniqueness.

The proof of \tref{thm: main} has been completed.


\end{document}